\documentclass[12pt,letterpaper]{amsart}
\usepackage[leqno]{amsmath}
\usepackage{amsfonts}
\usepackage{amssymb}
\usepackage{amsthm}
\usepackage{amssymb}
\usepackage[all]{xy}
\usepackage{graphicx}
\usepackage{multicol}

\usepackage{calc}

\theoremstyle{plain}
\newtheorem{theorem}{Theorem}[section]
\newtheorem{corollary}[theorem]{Corollary}
\newtheorem{proposition}[theorem]{Proposition}
\newtheorem{lemma}[theorem]{Lemma}
\newtheorem{conjecture}[theorem]{Conjecture}
\newtheorem{observation}[theorem]{Observation}
\newtheorem{problem}[theorem]{Problem}

\theoremstyle{definition}
\newtheorem{definition}[theorem]{Definition}

\newtheorem{remark}[theorem]{Remark}

\begin{document}

\begin{center}
{\Large \bf Upper bounds on the k-forcing number of a graph}\\[5ex]

\begin{multicols}{2}

David Amos\\
{\small Texas A \&\ M University\\}

\columnbreak

Yair Caro\\
{\small University of Haifa-Oranim\\}

\end{multicols}

\begin{multicols}{2}

Randy Davila\\
{\small Rice University\\}

\columnbreak

Ryan Pepper\footnote{corresponding author: pepperr@uhd.edu}\\
{\small University of Houston-Downtown\\}

\end{multicols}

\textbf{Abstract}

\end{center}


\noindent 

Given a simple undirected graph $G$ and a positive integer $k$, the $k$-forcing number of $G$, denoted $F_k(G)$, is the minimum number of vertices that need to be initially colored so that all vertices eventually become colored during the discrete dynamical process described by the following rule. Starting from an initial set of colored vertices and stopping when all vertices are colored: if a colored vertex has at most $k$ non-colored neighbors, then each of its non-colored neighbors becomes colored. When $k=1$, this is equivalent to the zero forcing number, usually denoted with $Z(G)$, a recently introduced invariant that gives an upper bound on the maximum nullity of a graph. In this paper, we give several upper bounds on the $k$-forcing number.  Notable among these, we show that if $G$ is a graph with order $n \ge 2$ and maximum degree $\Delta \ge k$, then $F_k(G) \le \frac{(\Delta-k+1)n}{\Delta - k + 1 +\min{\{\delta,k\}}}$.  This simplifies to, for the zero forcing number case of $k=1$, $Z(G)=F_1(G) \le \frac{\Delta n}{\Delta+1}$.  Moreover, when $\Delta \ge 2$ and the graph is $k$-connected, we prove that $F_k(G) \leq \frac{(\Delta-2)n+2}{\Delta+k-2}$, which is an improvement when $k\leq 2$, and specializes to, for the zero forcing number case, $Z(G)= F_1(G) \le \frac{(\Delta -2)n+2}{\Delta -1}$. These results resolve a problem posed by Meyer about regular bipartite circulant graphs.  Finally, we present a relationship between the $k$-forcing number and the connected $k$-domination number. As a corollary, we find that the sum of the zero forcing number and connected domination number is at most the order for connected graphs. 


\section{Introduction and Key Definitions}
Throughout this paper, all graphs are simple, undirected and finite. Let $G=(V,E)$ be a graph. We will use the basic notation: $n=n(G)=|V|$, $m=m(G)=|E|$, $\Delta(G)$, and $\delta(G)$; to denote respectively the order, size, maximum degree and minimum degree of $G$.  To denote the degree of a vertex $v$, we will write $deg(v)$.  The connected domination number, which we will denote with $\gamma_c = \gamma_c(G)$, is the size of a smallest set $S \subseteq V$ such that every vertex not in $S$ is adjacent to at least one vertex that is in $S$, and the subgraph induced by $S$ is connected (see \cite{CaroWestYuster, Chellali, Henning}). The $k$-independence number, which we will denote with $\alpha_k=\alpha_k(G)$, is the cardinality of a largest set $S \subseteq V$ such that the subgraph induced by $S$ has degree less than $k$ (so that $\alpha_1 = \alpha$; see \cite{AmosDavilaPepper, CaroHansberg, CaroPepper, Delavina3, Hansberg, Pepper}). Other necessary definitions will be presented throughout the paper as needed, and for basic graph theory definitions, the reader can consult \cite{West}. 

Now we introduce and define the $k$-forcing number.  Let $k \leq \Delta$ be a positive integer.  A set $S \subseteq V$ is a \emph{$k$-forcing set} if, when its vertices are initially colored (state 1) -- while the remaining vertices are intitially non-colored (state 0) -- and the graph is subjected to the following color change rule, all of the vertices in $G$ will eventually become colored (state 1). A colored vertex with at most $k$ non-colored neighbors will cause each non-colored neighbor to become colored.  The \emph{$k$-forcing number}, which we will denote with $F_k=F_k(G)$, is the cardinality of a smallest $k$-forcing set.  We will call the discrete dynamical process of applying this color change rule to $S$ and $G$ the \emph{$k$-forcing process}. If a vertex $v$ causes a vertex $w$ to change colors during the $k$-forcing process, we say that $v$ \emph{$k$-forces} $w$ (and we note here that a vertex can be $k$-forced by more than one other vertex). Our paper is about upper bounds on the $k$-forcing number.

This concept generalizes the recently introduced but heavily studied notion of the zero forcing number of a graph, which is denoted $Z=Z(G)$.  Indeed, $F_1(G)=Z(G)$, and throughout this paper, we will denote the zero forcing number with $F_1(G)$. The zero forcing number was introduced independently in \cite{AIM} and \cite{Burgarth}.  In \cite{AIM}, it is introduced to bound from below the minimum rank of a graph, or equivalently, to bound from above the maximum nullity of a graph.  Namely, if $G$ is a graph whose vertices are labeled from $1$ to $n$, then let $M(G)$ denote the maximum nullity over all symmetric real valued matrices where, for $i \ne j$, the $ij^{th}$ entry is nonzero if and only if $\{i,j\}$ is an edge in $G$, and zero otherwise. Then, the zero forcing number is an upper bound on $M(G)$, that is, $F_1(G)=Z(G) \ge M(G)$. In \cite{Burgarth}, it is indirectly introduced in relation to a study of control of quantum systems.  Besides its origins in the minimum rank/maximum nullity problem and the control of quantum systems, one can imagine other applications in the spread of opinions or disease in a social network (as described for a similar invariant by Dreyer and Roberts in \cite{Dreyer}). Some of the many other papers written about zero forcing number are: \cite{Chilakammari, Edholm, Eroh, Hogben, Meyer, Row2, Row, Yi}.

Although finding upper bounds on $F_{k}(G)$ in general was a goal, our real inspiration was to try and find upper bounds in terms of $\Delta(G)$ and the order of the graph. In finding such bounds we expose the sharp distinction in the $k$-forcing behavior of connected graphs with $\Delta=2$ and connected graphs with $\Delta=3$. This phenomenon motivates the study of finding an upper bound on $F_{k}(G)$ in terms of $k$ and $\Delta(G)$,  both in case that no connectivity constraint is imposed on $G$, and also in the  basic situation that $G$ is connected. Concentrating on $G$ being connected is of course justified by the simple observation that the $k$-forcing number is additive component-wise. 

The remainder of our paper is organized into as follows: Section 2 -- containing basic results and preliminaries; Section 3 -- containing upper bounds on $k$-forcing number with no connectivity conditions assumed; Section 4 -- containing improved upper bounds on $k$-forcing number for $k$-connected graphs; Section 5 -- containing upper bounds on $k$-forcing number for special families of graphs; and Section 6 -- containing conjectures and future work.  Before moving on, we now provide a few basic but useful facts about the $k$-forcing number. 

\section{Basic facts and preliminary results about $F_k(G)$}

It is easy to deduce that the zero forcing number is at least the minimum degree of the graph.  We generalize that idea with the following.

\begin{proposition}
Let $G=(V,E)$ be a graph with minimum degree $\delta$ and let $k$ be a positive integer. Then, $F_k(G) \geq \delta - k + 1$.
\end{proposition}
\begin{proof}
Let $S$ be a minimum $k$-forcing set of $G$ and color the vertices of $S$.  If $V-S = \emptyset$, then $|S|=F_k(G)=n$ and the inequality is trivially satisfied.  If $v \in V-S$ is an uncolored vertex that is $k$-forced by $w \in S$, then $w$ has at least $deg(w)-k \geq \delta -k$ colored neighbors.  Thus, counting $w$, there are at least $\delta - k +1$ vertices in $S$, completing the proof.
\end{proof}

It is well known that the zero forcing number of a complete graph $K_n$ is exactly $n-1$.  This observation generalizes as follows.

\begin{observation}
Let $G=K_n$ be a complete graph on $n$ vertices and let $k$ be a positive integer.  Then, $F_k(G) = \max\{n -k,1\}$.
\end{observation}

We note in passing that the converse of this observation is true when $k=1$, but this is not the case for $k \ge 2$. So, if $G$ is a graph with zero forcing number $F_1(G)=n-1$, then $G=K_n$. On the other hand, consider the complete bipartite graph $G=K_{3,3}$ and $k=2$.  For this graph, $F_2(G)=2=n-2$, but it is not $K_n$.

Finally, we can compare $k$-forcing numbers with different indices and easily compute some extreme cases with the following.

\begin{proposition}
Let $G=(V,E)$ be a graph. 
\begin{enumerate}
  \item If $G$ is connected with $n \geq 2$, then $F_\Delta(G) = 1$.
  \item If $G$ is connected with $\Delta \geq 2$, then $F_{\Delta - 1}(G) \le 2$ with equality holding if and only if $G$ is $\Delta$-regular.
  \item If $k$ is a positive integer, then $F_k(G) \ge F_{k+1}(G)$.
\end{enumerate}
\end{proposition}

\begin{proof}
To prove (1), assume $G$ is a connected graph with $n \geq 2$. Color a single vertex $v$ and notice that $v$ has at most $\Delta$ non-colored neighbors (by virtue of having at most $\Delta$ neighbors) and will therefore $\Delta$-force all of its non-colored neighbors to change color.  Since the same will be true for all the neighbors of $v$, the result now follows by connectivity of $G$. 

To prove (2), assume $G$ is a connected graph with $\Delta \geq 2$. Color a pair of adjacent vertices $u$ and $v$. Since each of those vertices now has at most $\Delta -1$ non-colored neighbors, they will both force all of their neighbors to change color. Since the same will be true for all of the newly formed pairs of colored adjacent vertices, the result now follows by connectivity of $G$. Namely, $F_{\Delta - 1}(G) \le 2$.  If $G$ is $\Delta$-regular, then coloring just one vertex cannot start the $(\Delta -1)$-forcing process.  This implies $F_{\Delta -1}(G) = 2$ when $G$ is $\Delta$-regular. If $G$ is not $\Delta$-regular, then let $v$ be a vertex with $deg(v)<\Delta$. Color $v$ initially and observe that $v$ has at most $\Delta -1$ non-colored neighbors.  Hence, $v$ will force all its neighbors to change color at the first step of the $(\Delta -1)$-forcing process, giving at least one pair of adjacent colored vertices which will then color the rest of the vertices by connectivity. Hence, if $G$ is not $\Delta$-regular, $F_{\Delta -1}(G)=1$, which finishes the argument.

To prove (3), that is to see that the $k$-forcing number is monotonically non-increasing, let $k$ be a positive integer and $S$ be a minimum $k$-forcing set. Observe that $S$ is also a $(k+1)$-forcing set so $F_{k+1}(G) \le |S| = F_k(G)$.

\end{proof}

From part (2) of the above proposition, we get several interesting and useful corollaries.  For example, we see that: if $G$ is cubic (that is $3$-regular) -- $F_2(G)=2$; if $G$ is $4$-regular -- $F_3(G)=2$; if $G$ is $5$-regular -- $F_4(G)=2$; etc. These, and item (2) above in general, are in sharp contrast to the $k$-forcing number when $k \leq \Delta-2$. In fact, $F_{\Delta -2}(G)$ can already be linear in $n$ when $\Delta =3$, as this is simply the zero forcing number of graphs with $\Delta =3$. As an example to illustrate this, attach a degree one vertex to each vertex of a path on $n \geq 3$ This results in a connected graph with $\Delta =3$ where the zero forcing number grows with $n$. Moreover, part (2) also gives us the well known results that the zero forcing number of paths is equal to $1$, while the zero forcing number of cycles is equal to $2$.

The following result appeared in \cite{Edholm} by Edholm, Hogben, Huynh, LaGrange, and Row, and will be used in subsequent proofs.

\begin{lemma} \cite{Edholm}\label{spread}
Let $G=(V,E)$ be a graph on $n\geq 2$ vertices.  Then,
\begin{enumerate}	
\item For $v\in V(G)$, $F_1(G)-1\leq F_1(G-v)\leq F_1(G)+1$,		
\item For $e\in E(G)$, $F_1(G)-1\leq F_1(G-e)\leq F_1(G)+1$.
\end{enumerate}
\end{lemma}

\section{upper bounds on $F_k(G)$ based on $n$, $\Delta$, $\delta$, and $k$}

In this section and the next, we present the main results of our paper. We need the following three lemmas.

\begin{lemma}\label{lem1}
Let $G=(V,E)$ be a graph on at least two vertices, with minimum degree $\delta\geq 1$ and let $S$ be a minimum $k$-forcing set of $G$ (namely $|S|=F_{k}(G)$). Then for every vertex $v\in S$ there are at least $\min\left\{deg(v), k \right\}$ vertices in $V- S$ adjacent to $v$.
\end{lemma}

\begin{proof}
Proceeding by contradiction, suppose on the contrary that $S$ is a minimum $k$-forcing set containing a vertex $v$ having at most $\min\left\{deg(v), k\right\}-1$ neighbors in $V- S$. Color the vertices of $S$.  Denote by $A \subseteq S$ the set of colored neighbors of $v$ and denote by $B \subseteq V-S$ the set of non-colored neighbors of $v$.  Note that $|B| \le \min\{deg(v),k\} -1$.  We now proceed in two cases.

\emph{Case 1: $deg(v)\leq k$.} In this case, notice that $|B| \le deg(v)-1$ which implies $|A| \ge 1$.  Let $S'=S-A$ be the vertices remaining in $S$ after all vertices of $A$ are taken out (that is, the coloring is removed). Now apply the $k$-forcing process to $S'$ and observe that after the first step, $v$ will have $k$-forced its entire neighborhood to change color.  Since this includes the vertices originally in $A$, all of $S$ will be colored, which then $k$-forces the rest of the graph because $S$ was a $k$-forcing set.  Clearly, $|S'|<|S|$ which contradicts the fact that $S$ was a minimum $k$-forcing set.

\emph{Case 2: $deg(v)>k$.} In this case, notice that $|B| \leq k-1 <k$.  Let $j$ be a positive integer such that $|B|=k-j$.  Thus, $|A|=deg(v)-|B|=deg(v)-k+j \ge j+1$, where the last inequality is true because $deg(v)-k>0$. Let $S'$ be the vertices remaining in $S$ after $j$ vertices from $A$ are taken out (that is, the coloring is removed). Observe that in $S'$, $v$ now has exactly $deg(v)-k$ neighbors.  Hence, $v$ has exactly $k$ non-colored neighbors in $V-S'$.  Apply the $k$-forcing process to $S'$ and observe that after the first step, $v$ will have $k$-forced its remaining non-colored neighbors to change color. Since this includes the $j$ vertices originally in $A$, all of $S$ will be colored, which then $k$-forces the rest of the graph because $S$ was a $k$-forcing set.  Clearly, $|S'|<|S|$ which contradicts the fact that $S$ was a minimum $k$-forcing set.
\end{proof} 

\begin{lemma}\label{lem2}
Let $G=(V,E)$ be a graph and let $S$ be a $k$-forcing set of $G$ with $k\geq 2$. If $w\in V-S$ and $deg(w) \geq k$, then $w$ has at least $k-1$ neighbors in $V-S$.
\end{lemma} 

\begin{proof}
Proceeding by contradiction, suppose there is $w\in V-S$ with $deg(w)\geq k\geq 2$, such that $w$ has at most $k-2$ neighbors in $V-S$. Now, the number of neighbors of $w$ in $S$ is at least $deg(w)-k+2\geq 2$ (as $deg(w)\geq k$). Let $u$ and $v$ be two neighbors of $w$ in $S$. Consider $D=( S\cup\{w\})- \{u,v\}$, with $|D|=|S|-1$, and color the vertices of $D$.  Now as $w$ is in $D$ and has at most $k$ non-colored neighbors in $V-D$ (having at most $k-2$ to start and having added $2$), it will $k$-force its neighborhood to change color. In particular, $w$ will $k$-force $u$ and $v$ producing a colored set of vertices containing $S$, which would then color all vertices because $S$ is a $k$-forcing set.  Hence, $D$ is a $k$-forcing set with smaller cardinality than $S$, a contradiction.
\end{proof}   

\begin{lemma}\label{lem3}
Let $G=(V,E)$ be a graph and let $S$ be a $k$-forcing set of $G$ with $k\geq 2$. If $w \in V-S$ and $2 \leq deg(w) <k$, then $w$ has at most one neighbor in $S$.
\end{lemma} 

\begin{proof}
Proceeding by contradiction, suppose there is $w \in V-S$ with $2 \leq deg(w) <k$, such that $w$ has at least two neighbors in $S$.  Let $u$ and $v$ be two neighbors of $w$ in $S$. Consider $D=(S\cup \{w\})- \{u,v\}$, with $|D|=|S|-1$, and color the vertices of $D$. Now as $w$ is in $D$ and has at most $k-1$ non-colored neighbors in $V-D$, it will $k$-force its neighborhood to change color. In particular, $w$ will $k$-force $u$ and $v$ producing a colored set of vertices containing $S$, which will then color all vertices because $S$ is a $k$-forcing set. Hence, $D$ is a $k$-forcing set with smaller cardinality than $S$, a contradiction.
\end{proof}   

We are now ready to present our main result, a sharp upper bound on the $k$-forcing number of a graph in terms of $k$, $\delta$, $\Delta$, and $n$. 
 
\begin{theorem}\label{main}
Let $k$ be a positive integer and let $G=(V,E)$ be a graph on $n\geq 2$ vertices with maximum degree $\Delta\geq k$ and minimum degree $\delta \geq 1$. Then,
\[
F_k(G) \leq \frac{(\Delta-k+1)n}{\Delta -k+1+\min\{\delta,k\}}.
\]

\end{theorem}
\begin{proof}
Suppose first $k=1$. Let $S$ be a minimum $1$-forcing set and $Q=V- S$. Denote the number of edges between $S$ and $Q$ by $e(S,Q)$. Each vertex $w\in Q$ has at most $\Delta$ neighbors in $S$. Hence, $|Q|\Delta\geq e(S,Q)$. By Lemma \ref{lem1}, for $k=1$, each vertex $v\in S$ has at least one neighbor in $Q$. Hence $e(S,Q)\geq |S|$. Combining these we get; 
\[|Q|\Delta\geq e(S,Q)\geq |S|, \]
which, after substituting $|Q|=|V-S|=n-|S|$, yields;
\[(n-|S|)\Delta\geq |S|.\] 
From this, after rearranging, we deduce,
\[ F_{1}(G)=|S|\leq \frac{n\Delta}{\Delta+1}. \] 
This proves the theorem when $k=1$.

Suppose now $k\geq 2$. Let $S$ be a minimum $k$-forcing set and $Q=V- S$. Denote the number of edges between $S$ and $Q$ by $e(S,Q)$. Partition $Q=W\cup X \cup Y$ into the three sets $W,X,$ and $Y$, defined by:
\[ W=\{ v\in Q \mid deg(v)\geq k \}; \] 
\[ X=\{ v \in Q \mid 2 \leq deg(v) <k \}; \] 
\[ Y = \{ v \in Q \mid deg(v)=1 \}. \]
By Lemma \ref{lem2}, each vertex $v \in W$ has at least $k-1$ neighbors in $Q$ and thus at most $deg(v) -k+1 \leq \Delta -k+1$ neighbors in $S$. Therefore, $W$ contributes at most $(\Delta-k+1)|W|$ edges to $e(S,Q)$. By Lemma \ref{lem3}, each vertex $v \in X$ has at most one neighbor in $S$. Therefore, since $\Delta \geq k$, $X$ contributes at most $|X| \leq (\Delta-k+1)|X|$ edges to $e(S,Q)$.  Clearly, each vertex $v \in Y$ has at most one neighbor in $S$. Therefore, again since $\Delta \geq k$, $Y$ contributes at most $|Y| \leq (\Delta-k+1)|Y|$ edges to $e(S,Q)$. Putting this together, and since $|Q|=|W|+|X|+|Y| = |V-S|=n-|S|$, we get the following upper bound:
\begin{equation}\label{upper}
e(S,Q)  \leq (\Delta-k+1)(|W|+|X|+|Y|) = (\Delta-k+1)(n-|S|).
\end{equation}
By Lemma \ref{lem1}, each vertex $v\in S$ has at least $\min\{deg(v), k\} \geq \min\{\delta,k\}$ neighbors in $Q$. Hence, we get the following lower bound:
\begin{equation}\label{lower}
e(S,Q) \geq |S|\min\{\delta,k\}.
\end{equation}
Combining the upper bound on $e(S,Q)$ from Inequality (\ref{upper}) with the lower bound on $e(S,Q)$ from Inequality (\ref{lower}), we find;
\[
|S|\min\{\delta,k\} \leq (\Delta-k+1)(n-|S|),
\]
which, after rearranging yields, $|S|\leq \frac{(\Delta-k+1)n}{\Delta-k+1+\min\{\delta, k\}}$, completing the proof since $|S|=F_k(G)$. 
\end{proof}

Next we present some corollaries to Theorem \ref{main}. First, by considering graphs with $\delta \geq k$, we can simplify the expression as seen below.

\begin{corollary}\label{kcor}
Let $k$ be a positive integer and let $G=(V,E)$ be a graph on $n \geq 2$ vertices with minimum degree $\delta \geq k$ and maximum degree $\Delta$. Then,
\[
F_k(G) \leq \frac{(\Delta-k+1)n}{\Delta+1},
\]
and this inequality is sharp.
\end{corollary}

To see that equality in the above corollary can be achieved, let $G=K_{\Delta+1}$, with $\Delta \ge k$, and consider any number of disjoint copies of $G$. By further specifying Theorem \ref{main} to the case $k=1$, that is, by focusing on the zero forcing number, we get the simplification below (put $k=1$ into Corollary \ref{kcor}). This result answers an open problem posed by Meyer in \cite{Meyer}, where it was asked whether there was an upper bound on zero forcing number for bipartite circulant graphs in terms of the maximum degree and the order. The corollary below answers this question in the affirmative in a much more general way.

\begin{corollary}\label{ratio}
Let $G=(V,E)$ be a graph with $\delta \geq 1$. Then,
\[
Z(G)=F_1(G) \leq \frac{\Delta}{\Delta+1}n,
\]
and this inequality is sharp.
\end{corollary}

Equality is achieved for disjoint unions of complete graphs of the same order.  This yields, for example, $Z(G)=F_1(G)\leq \frac{3n}{4}$ when $\Delta=3$. When one additionally imposes the condition of connectivity, we can get nice improvements as described in the next section.

\section{Improved upper bounds on $F_k(G)$ for connected graphs}

In this section, we add the condition that the graphs are connected and get improvements on Theorem \ref{main} in the case that $k\leq 2$ (including the case for zero forcing number). First, we will need the following result, which itself has as a corollary that the zero forcing number is bounded above by the order minus the connected domination number -- that is, $Z(G)=F_1(G)\leq n(G)-\gamma_c(G)$.

We need a few more definitions before proceeding. A \emph{$k$-dominating set} of a graph $G$ is a set $D$ of vertices such that every vertex not in $D$ is adjacent to at least $k$ vertices in $D$. The \emph{$k$-domination number}, denoted $\gamma_k(G)$ is the cardinality of a smallest $k$-dominating set (see the survey \cite{Chellali} for more on $k$-domination).  When $k=1$, this is the same as a dominating set and the domination number. A \emph{connected $k$-dominating set} is a $k$-dominating set whose vertices induce a connected subgraph. We will denote the cardinality of a smallest connected $k$-dominating set by $\gamma_{k,c}(G)$ and call this the \emph{connected $k$-domination number}. A graph $G$ is called \emph{$k$-connected} if it has more than $k$ vertices and the removal of any set (even an empty set) of fewer than $k$ vertices results in a connected subgraph.

\begin{lemma}\label{lem4}
Let $k$ be a positive integer and $G=(V,E)$ be a $k$-connected graph with $n > k$. If $S$ is a smallest $k$-forcing set such that the subgraph induced by $V-S$ is connected, then $V-S$ is a connected $k$-dominating set of $G$.
\end{lemma}
\begin{proof}
Proceeding by contradiction, let $S$ be a smallest $k$-forcing set such that the subgraph induced by $V-S$ is connected (that such a set exists can be seen by setting $S$ to be any collection of $n-1$ vertices, in which case $V-S$ is vacuously connected). Assume $V-S$ is not a $k$-dominating set. Let $A$ be the set of all vertices in $S$ which have less than $k$ neighbors in $V-S$, and let $B=S-A$. Since $V-S$ is not $k$-dominating, we know that $A \neq \emptyset$. 

Suppose first that there is a vertex $u \in A$ which has a neighbor $w\in B$. Since $w\in B$, it has at least $k\geq 1$ neighbors in $V-S$.  Let $x\in V-S$ be a neighbor of $w$. Now, let $S'=S-\{w\}$. Observe that $S'$ is a $k$-forcing set since, after coloring its vertices and applying the $k$-forcing process, $u$ will $k$-force $w$ on the first step -- as $u$ has at most $k$ non-colored neighbors.  Hence we will arrive at a colored set of vertices containing all of $S$, which will then $k$-force all of $G$ since $S$ is a $k$-forcing set. Moreover, the subgraph induced by $V-S'$ is connected since $V-S$ induces a connected graph and $w$ is adjacent to $x \in V-S$.  Therefore, $S'$ is a $k$-forcing set such that the subgraph induced by $V-S'$ is connected -- which is a contradiction since $S$ is the smallest such set and $|S'|<|S|$.

Now suppose that no vertex in $A$ is adjacent to any vertex in $B$. Since $G$ is $k$-connected, there must be a vertex in $A$ that has a neighbor in $V-S$. Let $v \in A$ be such a vertex and let $x \in V-S$ be adjacent to $v$. Either all neighbors of $v$ are in $V-S$, or else there is a vertex in $A$ adjacent to $v$. In the first case, notice that $deg(v)<k$.  This is a contradiction since $G$ is $k$-connected and the minimum degree of any $k$-connected graph is at least $k$. In the second case, let $u \in A$ be adjacent to $v$ and consider the set $S'=S-\{v\}$. Observe that $S'$ is a $k$-forcing set since $u$, having at most $k$ neighbors in $V-S'$, will $k$-force $v$ on the first step of the $k$-forcing process applied to $S'$. Hence we will arrive at a colored set of vertices containing all of $S$, which will then $k$-force all of $G$ since $S$ is a $k$-forcing set. Moreover, the subgraph induced by $V-S'$ is connected since $V-S$ induces a connected graph and $v$ is adjacent to $x \in V-S$. Therefore, $S'$ is a $k$-forcing set such that the subgraph induced by $V-S'$ is connected -- which is a contradiction since $S$ is the smallest such set and $|S'|<|S|$. Since we arrive at a contradiction in all cases, the proof is complete.
\end{proof}

\begin{corollary}\label{connected_dom}
Let $k$ be a positive integer and let $G=(V,E)$ be a $k$-connected graph with $n>k$ vertices, $k$-forcing number $F_k(G)$, and connected $k$-domination number $\gamma_{k,c}(G)$. Then,
\[
F_k(G)\leq n-\gamma_{k,c}(G).
\]
\end{corollary}
\begin{proof}
Let $S$ be a smallest set which is $k$-forcing and such that $V-S$ induces a connected subgraph (that such a set exists can be seen by setting $S$ to be any collection of $n-1$ vertices, in which case $V-S$ is vacuously connected) . Thus, from Lemma \ref{lem4}, $V-S$ is a connected $k$-dominating set of $G$. So, $F_k(G)\le |S|$ and $\gamma_{k,c}(G) \le |V-S| = n -|S|$. From these, after rearranging, we deduce our result, $F_k(G) \le |S| \leq n - \gamma_{k,c}(G)$.
\end{proof}

From this, we deduce the very appealing result below which relates the zero forcing number to the connected domination number, both NP-hard graph invariants, in an incredibly simple way. From it, any lower bound on the connected domination number will give an upper bound on the zero forcing number, and any lower bound on the zero forcing number will give an upper bound on the connected domination number.  Since connected domination number has been more widely studied, this gives the researchers on zero forcing number new avenues to pursue.

\begin{corollary}\label{connected_dom2}
Let $G=(V,E)$ be a connected graph with $n\geq 2$, connected domination number $\gamma_c(G)$ and zero forcing number $Z(G)=F_1(G)$. Then,
\[
Z(G)=F_1(G) \leq n - \gamma_c(G),
\]
and this inequality is sharp.
\end{corollary}
\begin{proof}
Let $k=1$ in Corollary \ref{connected_dom} to get the bound. To see that the bound is sharp, consider complete graphs and cycles.
\end{proof}

\begin{theorem}\label{main2}
Let $k$ be a positive integer and let $G=(V,E)$ be a $k$-connected graph with $n>k$ vertices and $\Delta \ge 2$. Then,
\[
F_k(G) \leq \frac{(\Delta-2)n+2}{\Delta+k-2},
\]
and this inequality is sharp.
\end{theorem}
\begin{proof}
Let $S$ be a smallest set which is $k$-forcing and such that $V-S$ induces a connected subgraph. Thus, from Lemma \ref{lem4}, $Q=V-S$ is a connected $k$-dominating set. We will denote the number of edges between $S$ and $Q$ by $e(S,Q)$.  Since $Q$ is $k$-dominating, each vertex in $S$ is adjacent to at least $k$ vertices of $Q$. This gives us the following lower bound on $e(S,Q)$.
\begin{equation}\label{L}
e(S,Q) \geq k|S|
\end{equation}
Now, let $v \in Q$ and denote by $deg_Q(v)$ the number of neighbors $v$ has in $Q$. With this notation, $deg(v)-deg_Q(v)$ is the number of neighbors $v$ has in $S$. Consequently, summing over all the vertices in $Q$,
\begin{equation}\label{U}
e(S,Q) = \sum_{v\in Q} (deg(v)-deg_Q(v)) = \sum_{v\in Q}deg(v) - \sum_{v\in Q}deg_Q(v).
\end{equation}
Let $[Q]$ denote the subgraph induced by $Q$, and note that, since $[Q]$ is connected, $m([Q]) \geq |Q| -1$.  Furthermore,
\[
\sum_{v\in Q}deg_Q(v) =2m([Q]) \geq 2|Q|-2.
\]
Combining this with Inequality (\ref{U}), 
\[
\sum_{v\in Q}deg(v)-\sum_{v\in Q}deg_Q(v) \leq \Delta |Q| - 2|Q| + 2,
\]
which yields the following upper bound on $e(V,Q)$;
\[
e(S,Q) \leq \Delta |Q| - 2|Q|+2.
\]
Finally, we combine this with the lower bound from Inequality (\ref{L}), and substitute $|Q|=n-|S|$, to get:
\[
k|S| \leq (\Delta -2)(n-|S|)+2.
\]
This inequality can be solved for $|S|$, resulting in,
\[
|S| \leq \frac{(\Delta -2)n+2}{\Delta +k-2}.
\] 
Hence, because $F_k(G) \leq |S|$, the inequality is proven. To see that the inequality is sharp, consider a complete graph $K_{\Delta +1}$ and $k=1$ or $k=2$. 
\end{proof}

\begin{remark}
Theorem \ref{main2} is an improvement over Theorem \ref{main} and Corollary \ref{kcor} when the graph is $k$-connected, $\delta \geq k$, and $k \leq 2$.  However, when $k \geq 3$, even with the $k$-connected condition, Theorem \ref{main} is superior. 
\end{remark}

The corollary below has the nice special case that for connected graphs with $\Delta=3$, we have $Z(G)=F_1(G)\leq \frac{n}{2}+1$.

\begin{corollary}\label{cor3}
Let $G=(V,E)$ be a connected graph with $\Delta \geq 2$. Then,
\[
Z(G)=F_1(G) \leq \frac{(\Delta -2)n+2}{\Delta -1},
\]
and this inequality is sharp.
\end{corollary}
\begin{proof}
For the inequality, set $k=1$ in Theorem \ref{main2}. To see that equality can be achieved, consider complete graphs and complete bipartite graphs with equal parts.
\end{proof}

The next theorem combines Corollary \ref{connected_dom2} and Corollary \ref{cor3} into one result. It shows that Corollary \ref{connected_dom2} is always as good as Corollary \ref{cor3}, but Corollary \ref{cor3} gives an efficiently computable upper bound on the zero forcing number while Corollary \ref{connected_dom2} involves an NP-hard graph invariant, $\gamma_c(G)$. 

\begin{theorem}
Let $G=(V,E)$ be a connected graph with $\Delta \geq 2$. Then,
\[
Z(G) = F_1(G) \leq n- \gamma_c(G) \leq \frac{(\Delta-2)n+2}{\Delta -1}.
\]
\end{theorem}
\begin{proof}
Let $T$ be a spanning tree of $G$ with the maximum number of leaves, and denote the number of leaves in $T$ with $L$. It is well known that $\gamma_c(G)=n-L$, see \cite{CaroWestYuster}, so that $n-\gamma_c(G)=L$. Now, let $D \leq \Delta$ be the maximum degree of $T$, and let $n_i$ be the number of vertices in $T$ of degree $i$. Since $m(T)=n-1$, $L=n_1$, $n=n_1+n_2+\ldots +n_D$, and
\[
2m(T) = \sum_{i=1}^D in_i = n_1 + \sum_{i=2}^D in_i,
\]
we deduce the following.
\begin{equation}\label{C}
L(D-1) + 2n -2 = L(D-1) + 2m(T) = Dn_1 - n_1 + n_1 + \sum_{i=2}^D in_i \leq Dn. 
\end{equation}
Comparing the left- and right-most expressions in the above inequality, $L(D-1)+2n-2 \leq Dn$, and solving for $L$ (noting that $D \geq 2$, since the condition $\Delta \geq 2$ implies $n \geq 3$, and a spanning tree on at least three vertices has maximum degree at least two), we get:
\[
L \leq \frac{(D-2)n+2}{D-1} \leq \frac{(\Delta-2)n+2}{\Delta -1},
\]
by monotonicity, which completes the proof.
\end{proof}

\begin{remark}
Of course, this result can be thought of as a lower bound for connected domination number. Namely, if $G=(V,E)$ is a connected graph with $\Delta \geq 2$, then,
\[
\gamma_c(G) \geq \frac{n-2}{\Delta -1}.
\]
This inequality is sharp since equality is achieved for paths, cycles, and any graph with $\Delta = n-1$.
\end{remark}

\section{Upper bounds on $F_k(G)$ for special families of graphs}

\subsection{Graphs that are Hamiltonian and Cycle-Trees}

First, for the case $k=1$, we give an upper bound on the zero forcing number $Z(G)=F_1(G)$ when $G$ is Hamiltonian, where a Hamiltonian graph is a graph which has a cycle containing all the vertices. Given a Hamiltonian graph $G$ with a Hamiltonian cycle $C$, a chord is any edge not on $C$. Thus, with these definitions, Hamiltonian graphs with no chords are isomorphic to $C_n$ -- prompting us to focus on those Hamiltonian graphs with at least one chord.
 
\begin{theorem}
Let $G=(V,E)$ be a Hamiltonian graph with $t\geq 1$ chords and $n \geq 4$. Then,
\[
F_{1}(G)\leq t+1.
\]
\end{theorem}

\begin{proof}
We will proceed by induction on $t$. If $t=1$, take one vertex from the chord and a degree $2$ vertex adjacent to it, and color those two vertices.  The degree $2$ vertex will force all degree $2$ vertices around the cycle until reaching the second vertex of the chord, which then has at most one non-colored neighbor, and so forces all the remaining vertices. Hence, $Z(G)=F_{1}(G)=2=t+1$ -- settling the base case. Assume that the theorem is true for all Hamiltonian graphs with $t \geq 1$ chords, and let $G$ be a Hamiltonian graph with $t+1$ chords. Consider $G-e$, where $e$ is a chord of $G$. Since, $G-e$ is still Hamiltonian and has $t$ chords, $F_1(G-e) \leq t +1$ by inductive assumption.  Using Lemma \ref{spread}, $F_{1}(G)-1\leq F_{1}(G-e)\leq t+1$. Hence, $F_{1}(G)\leq (t+1)+1$, and the general result follows by induction.
\end{proof}

This result is an improvement over Corollary \ref{cor3} when $\Delta=3$, as seen below.

\begin{corollary}
Let $G=(V,E)$ be a Hamiltonian graph on $n$ vertices such that $\Delta(G)=3$ and suppose $G$ has $n_{3} \geq 2$ vertices of degree $3$. Then,
\[
F_{1}(G)\leq \frac{n_{3}}{2}+1\leq \frac{n}{2}+1.
\]
\end{corollary} 
\begin{proof}
Observe that the number of chords $t$ satisfies, $t=\frac{n_{3}}{2}\leq \frac{n}{2}$.
\end{proof}

Next we present a related idea for graphs that have $2$-factors, where a $k$-factor is a $k$-regular spanning subgraph. This includes some Hamiltonian graphs, slightly expanding our scope. We need the following definition. 

\begin{definition}
A connected graph $G=(V,E)$ will be called a \emph{cycle-tree} if it consists of $q$ vertex disjoint cycles that are connected by $q-1$ edges. Thus, a cycle-tree with $n$ vertices will have $m=n+q-1$ edges and each edge between two cycles is a cut-edge. 
\end{definition}

\begin{theorem}\label{cycle-tree}
Let $G=(V,E)$ be a cycle-tree with $q \geq 1$ cycles. Then,
\[
Z(G)=F_{1}(G)\leq 2q.
\]
\end{theorem}
\begin{proof}
Proceeding by induction, let $q=1$. In this case, $G$ is a cycle and hence satisfies, $Z(G)=F_1(G)=2=2q$, which settles the base case. Now assume the theorem is true for all cycle-trees with less than $q \geq 2$ cycles. Let $C$ be an end-cycle, that is a cycle connected to the rest of the graph by a unique edge, $e=\{u,v\}$, where $u \in C$ and $v \in V-C$ (by end-cycle we mean the analog of a leaf in a tree). Now, the subgraph induced by $V-C$, which we will denote by $[V-C]$, is a cycle-tree with $q-1<q$ cycles. Therefore, $[V-C]$ satisfies the inequality by inductive assumption. This gives,
\[
Z([V-C])=F_1([V-C])\leq 2(q-1) = 2q-2.
\]
Let $S$ be a minimum $1$-forcing set of $[V-C]$, let $w$ be a neighbor of $u$ on $C$, and consider the set $T=S \cup \{u,w\}$. Now, if $v \in S$, then since $\{u,w\}$ clearly $1$-forces all of $C$, and the $1$-forcing process in $[V-C]$ is not affected, $T$ is a $1$-forcing set of $G$. On the other hand, if $v \notin S$, $v$ is $1$-forced by one of its neighbors in $[V-C]$, leading to the same conclusion. Hence, $T$ is a $1$-forcing set of $G$. Finally, 
\[
Z(G)=F_1(G)\leq |T| = |S \cup \{u,w\}|=|S|+|\{u,w\}|=F_1([V-C])+2 = 2q,
\]
and the general result follows by induction.
\end{proof} 

\begin{remark}
If $\Delta=3$ and $n \geq 4q-2$, this result is an improvement over Corollary \ref{cor3}. For example, for cycle-trees with $\Delta=3$ that are triangle-free (or even for those with at most $2$ triangles), Theorem \ref{cycle-tree} is at least as good of a bound as Corollary \ref{cor3} and Corollary \ref{ratio}, since:
\[
Z(G)=F_1(G) \leq 2q \leq \frac{n}{2}+1 = \frac{(\Delta-2)n+2}{\Delta -1} \leq \frac{\Delta}{\Delta+1}n = \frac{3n}{4}.
\]
\end{remark}

\subsection{Trees}

Some of our results can be slightly improved by restricting our attention to trees. First, observe that for trees it is well known that $Z(T)=F_{1}(T)=P(T)$, where $P(T)$ is the path-cover number of $T$, that is the minimum number of vertex disjoint induced paths that cover all vertices of $T$ \cite{Eroh}. The result below shows that the zero forcing number for trees can be bounded above and below by linear functions of the number of leaves.

\begin{theorem}
Let $T$ be a tree with $n\geq 2$ and $n_1$ leaves. Then,
\[
\lceil\frac{n_{1}}{2}\rceil\leq P(T)=F_{1}(T)\leq n_{1}-1,
\]
and both extremes are sharp.
\end{theorem}

\begin{proof}
For the lower bound, notice that every path in a path cover can occupy at most two leaves, hence $F_{1}(T)=P(T)\geq \lceil \frac{n_{1}}{2}\rceil$. To see that equality can be achieved, attach two leaves to each vertex of a path, and color one vertex from each pair of leaves. These leaves now $1$-force their neighbors on the path, which in turn $1$-force the remaining leaves. Consequently, they form a $1$-forcing set and,
\[
\frac{n_{1}}{2}\geq F_{1}(T)=P(T)\geq \lceil\frac{n_{1}}{2}\rceil=\frac{n_{1}}{2}.
\]

For the upper bound, we will proceed by induction on $n_1$. If $n_{1}=2$, then $T$ is a path and clearly, $F_{1}(T)=P(T)=1=n_{1}-1$, which settles the base case. Assume the theorem is true for all trees with less than $n_1 \geq 3$ leaves. Suppose there is a support vertex $w$ with at least two leaves adjacent to it, and denote by $u$ and $v$ two leaves adjacent to $w$. Let $T'=T-v$ be the tree with $n_1'=n_1-1$ leaves that results from deleting $v$ from $T$. Then, using our inductive assumption and appealing to Lemma \ref{spread},
\[
F_1(T)-1 \leq F_1(T') = P(T') \leq n_{1}'-1=n_1-2. 
\]
From this we deduce that $F_1(T)=P(T)\leq n_1-1$.

On the other hand, suppose there is no vertex with at least two leaves. Together with the fact that $n_{1}\geq 3$, this implies there is a branch point in $T$ (a vertex of $T$ with degree at least three). Let $v$ be a leaf and $z$ the closest branch point to $v$. Let $T'$ be the tree formed by removing all vertices of the path between $v$ and $z$, excluding $z$ itself (this could even be $v$ alone). Since any path cover of $T'$ can be extended to a path cover of $T$ by adding at most one path, we have $F_1(T)=P(T) \leq P(T')+1$.  Combining this with our inductive assumption, and letting $n_1'=n_1-1$ be the number of leaves in $T'$, we complete the proof as follows,
\[
F_1(T)=P(T) \leq P(T')+1 \leq n_1'-1+1=n_1-1.
\]
To see that equality can be achieved, subdivide each edge in a star any number of times (even zero). Now color all but one leaf. Notice that each colored leaf will $1$-force all the vertices along the path to the center of the star, which will then $1$-force the vertices along the uncolored path towards the remaining leaf. For this family of trees, 
\[
F_1(T)=P(T)=n_1 -1.
\]
\end{proof}

We can now get a slight improvement over Corollary \ref{cor3} for trees.

\begin{corollary}
Let $T$ be a tree with $n$ vertices and $\Delta \geq 2$. Then,
\[
F_{1}(T)\leq \frac{(\Delta-2)n+2}{\Delta-1}-1.
\]
\end{corollary}

\begin{proof}
Among all trees on $n$ vertices, the ones with vertex degrees of $\Delta$ and $1$ only will be the ones with the maximum number of leaves. For these trees, $n_{\Delta}=n-n_1$ and $2n-2=\Delta n_{\Delta}+n_1$, where $n_{\Delta}$ means the number of vertices of degree $\Delta$.  Substituting the first equation into the second and solving for $n_1$, we find that the maximum number of leaves among all trees with $n$ vertices and maximum degree $\Delta$ is $\frac{(\Delta-2)n+2}{\Delta-1}$. Hence, for $T$, which may have fewer leaves,
\[
n_1 \leq \frac{(\Delta-2)n+2}{\Delta-1}.
\]
The result now follows by substitution of this inequality into the upper bound from the previous theorem.
\end{proof}

\subsection{Relationships between $k$-independence number and $k$-forcing number for $K_{1,r}$-free graphs}

A graph is $K_{1,r}$-free when it has no induced subgraph isomorphic to $K_{1,r}$.  When $r=3$, $K_{1,3}$-free graphs are sometimes called claw-free graphs. In this subsection we shall present upper bounds on $k$-forcing number for $K_{1,r}$-free graphs. 
 
\begin{proposition}\label{prop}
Let $G=(V,E)$ be a $K_{1,r}$-free graph with $\delta \geq 1$, where $k$ and $r \geq 3$ are positive integers. If $I$ is a maximum $k$-independent set, then every vertex in $V-I$ has at most $k(r-1)$ neighbors in $I$.
\end{proposition}
\begin{proof}
Proceeding by contradiction, let $I$ be a maximum $k$-independent set and assume $u \in V-I$ has at least $k(r-1)+1$ neighbors in $I$. Denote the set of neighbors of $u$ in $I$ by $S$, and note that $|S| \geq k(r-1)+1$. As $S$ is a subset of a $k$-independent set, it is also a $k$-independent set.  Hence, $\Delta ([S]) \leq k-1$, where $[S]$ denotes the subgraph induced by the vertices in $S$. Now, using a well known lower bound on independence number, we combine these ideas to get:
\[
\alpha ([S]) \geq \frac{|S|}{\Delta ([S]) +1} \geq \frac{k(r-1)+1}{\Delta([S]) +1} \geq \frac{k(r-1)+1}{k}=r-1 + \frac{1}{k} > r-1.
\]
Therefore, $\alpha([S]) \geq r$. Take any set of $r$ independent vertices in $[S]$ and together with $u$, they form a $K_{1,r}$, which is a contradiction. 
\end{proof}

\begin{theorem}
Let $G=(V,E)$ be a $K_{1,r}$-free graph with $\delta \geq 1$, where $k$ and $r \geq 3$ are positive integers. Then,
\[
F_{k(r-1)}(G) \leq n - \alpha_k(G).
\]
\end{theorem}
\begin{proof}
Let $I$ be a maximum $k$-independent set. By Proposition \ref{prop}, every vertex in $V-I$ has at most $k(r-1)$ neighbors in $I$. Color the vertices of $V-I$ and notice that after the first step of the $k(r-1)$-forcing process, all of $G$ will be colored. Hence, $V-I$ is a $k(r-1)$-forcing set and $F_{k(r-1)}(G) \leq |V-I| = n - \alpha_k(G)$.
\end{proof}

This gives the following result for the independence number when $k=1$.

\begin{corollary}
Let $G=(V,E)$ be a $K_{1,r}$-free graph with $\delta \geq 1$, where $r\geq 3$ is an integer. Then,
\[
F_{r-1}(G) \leq n - \alpha(G).
\]
\end{corollary}
 
Additionally, when $r=3$, that is the claw-free case, we get the following bound.

\begin{corollary}
Let $G=(V,E)$ be a $K_{1,3}$-free graph (a claw-free graph) with $\delta \geq 1$. Then, for every positive integer $k$,
\[
F_{2k}(G) \leq n - \alpha_k(G).
\]
\end{corollary}

\section{Conjectures, open problems, and future research}

Implicit as open problems for all the results in our paper, are the characterizations of the cases of equality.  In particular, we make the following conjecture about Corollary \ref{cor3}.

\begin{conjecture}
Let $G=(V,E)$ be a connected graph with $\Delta \ge 2$. Then,
\[
Z(G)=F_1(G) = \frac{(\Delta-2)n+2}{\Delta-1}
\]
if and only if $G=K_{\Delta +1}$ or $G=K_{\Delta,\Delta}$.
\end{conjecture}

We are also particularly interested in the characterization of the case of equality for Corollary \ref{connected_dom2}. We mention that equality holds for $K_n$, for $C_n$, and for $K_{p,q}$ with $p \geq q \geq 2$, but offer no conjecture.
\begin{problem}
Characterize the connected graphs with $n \geq 2$ for which,
\[
Z(G)=F_1(G)=n - \gamma_c(G).
\]
\end{problem}

In the future, we would like to study the $k$-forcing number for other families of simple graphs, work on improvements of our results when more structural information is assumed, and work on characterizations of the cases of equality for our bounds.

\end{document}